\documentclass[12pt]{article}

\usepackage{mathrsfs}

\usepackage{bbm}

\usepackage{verbatim}

\usepackage{amsthm}

\usepackage{amsmath}

\usepackage{amssymb}

\usepackage{graphicx}

\usepackage{epstopdf}

\usepackage{titling}

\usepackage{enumerate}

\usepackage{algpseudocode}

\usepackage{algorithm}

\usepackage{algorithmicx}

\usepackage{tikz}

\usepackage{hyperref}

\usepackage{enumitem}

\usepackage{cleveref}

\usepackage{MnSymbol}

\newtheorem{theorem}{Theorem}[section]

\newtheorem{prop}[theorem]{Proposition}

\newtheorem{lem}[theorem]{Lemma}

\newtheorem{cor}[theorem]{Corollary}

\newtheorem{observation}[theorem]{Observation}

\newtheorem{conj}[theorem]{Conjecture}

\newtheorem{rem}[theorem]{Remark}

\newlist{problems}{enumerate}{10}

\setlist[problems]{label*=\arabic*}

\crefname{problemsi}{problem}{problems}

\Crefname{problemsi}{Problem}{Problems}

\begin{document}

\title{On the number of $4$-cycles in a tournament}

\author{Nati Linial\thanks{School of Computer Science and engineering, The Hebrew University of Jerusalem, Jerusalem 91904, Israel. Email: {\tt nati@cs.huji.ac.il}. Research supported in part by grant 1169/14 of the Israel Science Foundation.}\and Avraham Morgenstern\thanks{Einstein Institute of mathematics, The Hebrew University of Jerusalem, Jerusalem 91904, Israel. Email: {\tt avraham.morgenstern@mail.huji.ac.il}}}


\maketitle

\setcounter{page}{1}

\vspace{-2em}

\begin{abstract}

If $T$ is an $n$-vertex tournament with a given number of $3$-cycles, what can be said about the number of its $4$-cycles? The most interesting range of this problem is where $T$ is assumed to have $c\cdot n^3$ cyclic triples for some $c>0$ and we seek to minimize the number of $4$-cycles. We conjecture that the (asymptotic) minimizing $T$ is a random blow-up of a constant-sized transitive tournament. Using the method of flag algebras, we  derive a lower bound that almost matches the conjectured value. We are able to answer the easier problem of maximizing the number of $4$-cycles. These questions can be equivalently stated in terms of transitive subtournaments. Namely, given the number of  transitive triples in $T$, how many transitive quadruples can it have? As far as we know, this is the first study of inducibility in tournaments.

\end{abstract}

\section{Introduction and notation}\label{sect:1}

\subsection{Notation}

For tournaments $T,H$, let $\mathbf{pr}(H,T)$ be the probability that a random set of $|H|$ vertices in $T$ spans a subtournament isomorphic to $H$. For an infinite family of tournaments $\cal T$, let $\mathbf{pr}(H,{\cal T})=\lim_{T\in{\cal T}, |T|\to\infty}\mathbf{pr}(H,T)$, assuming the limit exists. (Nonexistence of the limit may be repaired, of course, by passing to an appropriate subfamily).

We denote the transitive $m$-vertex tournament by $T_m$, and the $3$-vertex cycle by $C_3$. There are four isomorphism types of $4$-vertex tournaments, see~\Cref{fig:5}.

\begin{itemize}

\item

$C_4$ which is characterized by having a directed $4$-cycle,

\item

The transitive $T_4$,

\item

$W$, a cyclic triangle and a sink,

\item

$L$, a cyclic triangle and a source.

\end{itemize}

\begin{figure}

  \centering

  \includegraphics[width=160mm,height=40mm]{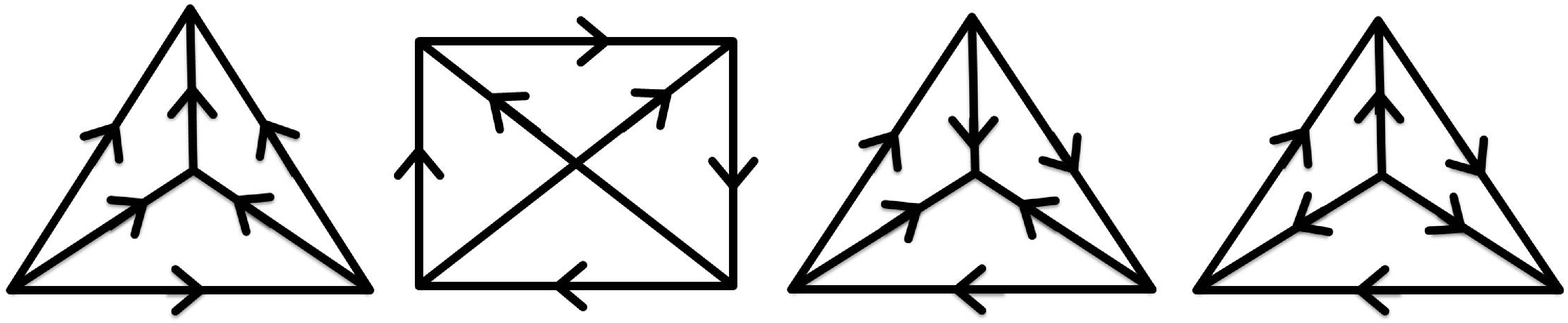}

  \caption{$T_4,C_4,W,L$}

  \label{fig:5}

\end{figure}

Denote $r({\cal T})=\mathbf{pr}(R,{\cal T})$ for any letter $r\in\{c_3,c_4,t_k,w,l\}$ (e.g., $c_3({\cal T})$ is the limit proportion of cyclic triangles in members of $\cal T$). We omit $\cal T$ when appropriate. We will always restrict ourselves to families for which all the relevant limits exist, though we do not bother to mention this any further.

In~\cite{LM} we initiated the study of $4$-local profiles of tournaments, namely the set $${\cal P}=\{(t_4({\cal T}),c_4({\cal T}),w({\cal T}),l({\cal T})) | {\cal T}\text{ is a family of tournaments for which all the limits exist}\}\subseteq\mathbb{R}^4.$$

Here we continue with these investigations.

\subsection{Our questions}

In studying the set $\cal P$ of $4$-local profiles of tournaments, it is of interest to understand its projection to the first two coordinates, which raises~\Cref{it:c,it:f} below. We are, in fact, interested in all the following six problems, but as we show below, they are interdependent.

\begin{problems}

\item\label{it:a} Maximize $c_4({\cal T})$ when $c_3({\cal T})$ is set.

\item\label{it:b} Maximize $t_4({\cal T})$ when $t_3({\cal T})$ is set.

\item\label{it:c} Maximize $c_4({\cal T})$ when $t_4({\cal T})$ is set.

\item\label{it:d} Minimize $c_4({\cal T})$ when $c_3({\cal T})$ is set.

\item\label{it:e} Minimize $t_4({\cal T})$ when $t_3({\cal T})$ is set.

\item\label{it:f} Minimize $c_4({\cal T})$ when $t_4({\cal T})$ is set.

\end{problems}

\begin{prop}\label{prop:equiv}

\Cref{it:a,it:b,it:c} are equivalent in the sense that the solution of any one of them can be transformed into a solution for the other two. Likewise \Cref{it:d,it:e,it:f} are equivalent.

\end{prop}

To prove this proposition we need

\begin{observation}\label{obs:43} $t_4-c_4=1-4c_3$ \end{observation}

\begin{proof}

We count cyclic triangles in $4$-vertex tournaments. There are none in $T_4$, two in $C_4$ and one each in $W$ and $L$. Therefore the number of cyclic triangles in an $n$-vertex tournament satisfies $c_3{n\choose 3}=\frac{2c_4+w+l}{n}{n\choose 4}$. The claim follows, since $t_4+c_4+w+l=1$.

\end{proof}

We can now prove Proposition~\ref{prop:equiv}.

\begin{proof}[Proof of Proposition~\ref{prop:equiv}]

Obviously, $t_3+c_3=1$. Combined with Observation~\ref{obs:43} this already proves that \Cref{it:a,it:b} and \Cref{it:d,it:e} are equivalent. To see that \Cref{it:e,it:f} are equivalent, note that \Cref{it:e} is equivalent to maximizing $t_3$ for given $t_4$, or, equivalently, to minimizing $c_3$ given $t_4$. The equivalence follows by Observation~\ref{obs:43}. A similar argument proves that \Cref{it:b,it:c} are equivalent.   

\end{proof}

\Cref{it:a,it:b,it:c} are rather straightforward and we proceed to solve them. \Cref{it:d,it:e,it:f} are deeper. By the equivalence proved above, the discussion is restricted to \cref{it:d}. We state a conjecture on the solution of this problem and prove a lower bound. This problem raises interesting structural limitations on tournaments, on which we elaborate in Section~\ref{sect:2}. We defer the technical proofs to Section~\ref{sect:proofs} and in Section~\ref{sect:further} we offer some further directions.

The three regions $\{(t_3,t_4)\},~\{(c_3,c_4)\},~\{(t_4,c_4)\}$ of the realizable pairs of parameters are illustrated in \Cref{fig:1,fig:2,fig:3}.

\begin{figure}

  \centering

  \includegraphics[width=160mm,height=100mm]{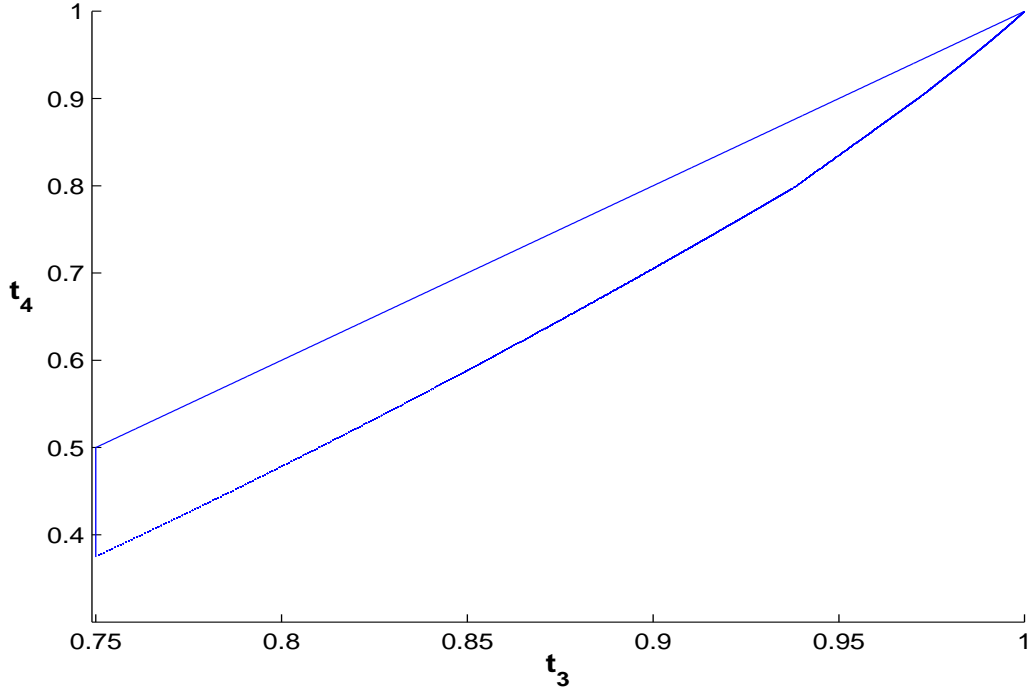}

  \caption{The boundary of the set $\{(t_3({\cal T}),t_4({\cal T}))\}$. The lower curve is conjectured.}

  \label{fig:1}

\end{figure}

\begin{figure}

  \centering

  \includegraphics[width=160mm,height=100mm]{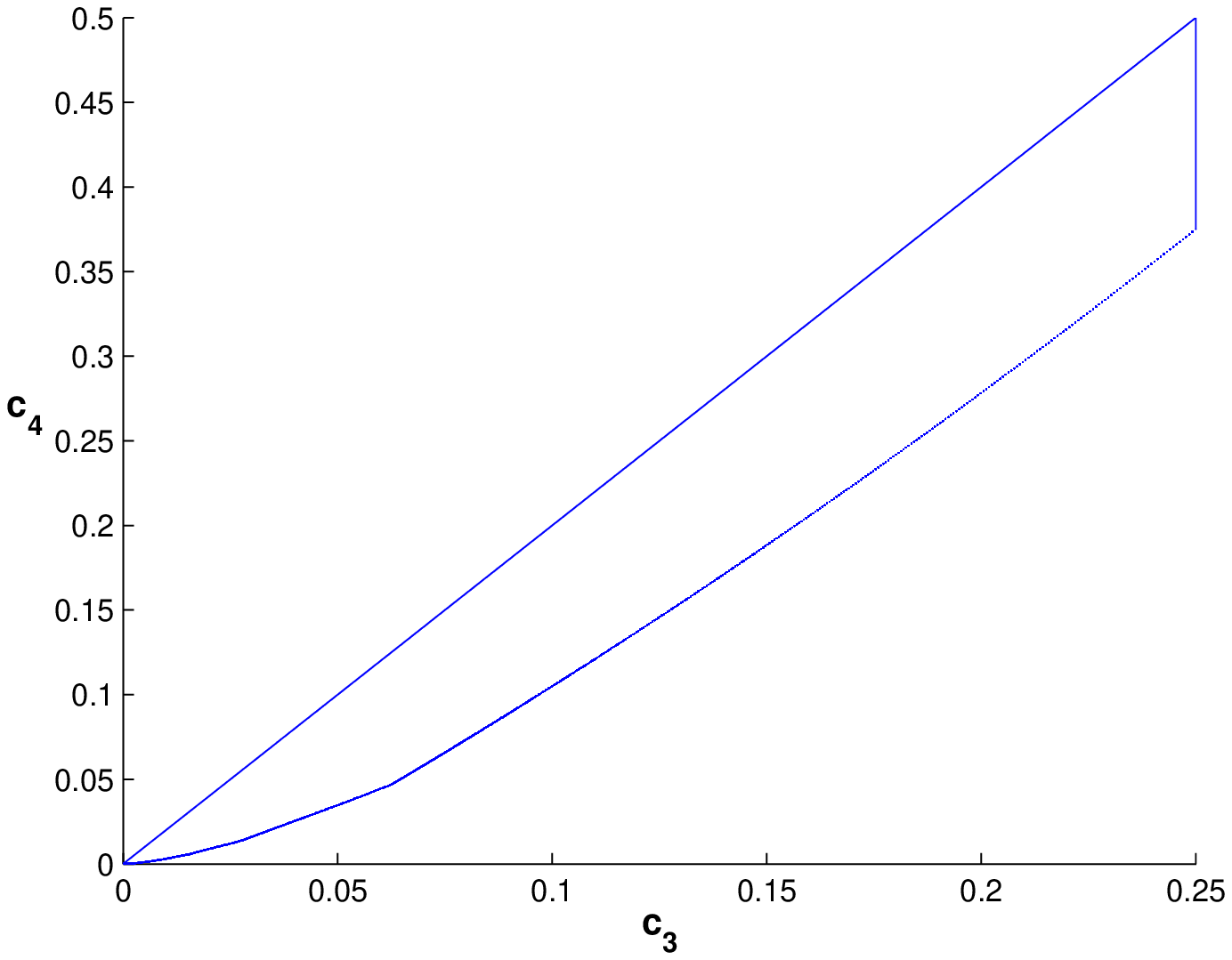}

  \caption{The boundary of the set $\{(c_3({\cal T}),c_4({\cal T}))\}$. The lower curve is conjectured.}

  \label{fig:2}

\end{figure}

\begin{figure}

  \centering

  \includegraphics[width=160mm,height=100mm]{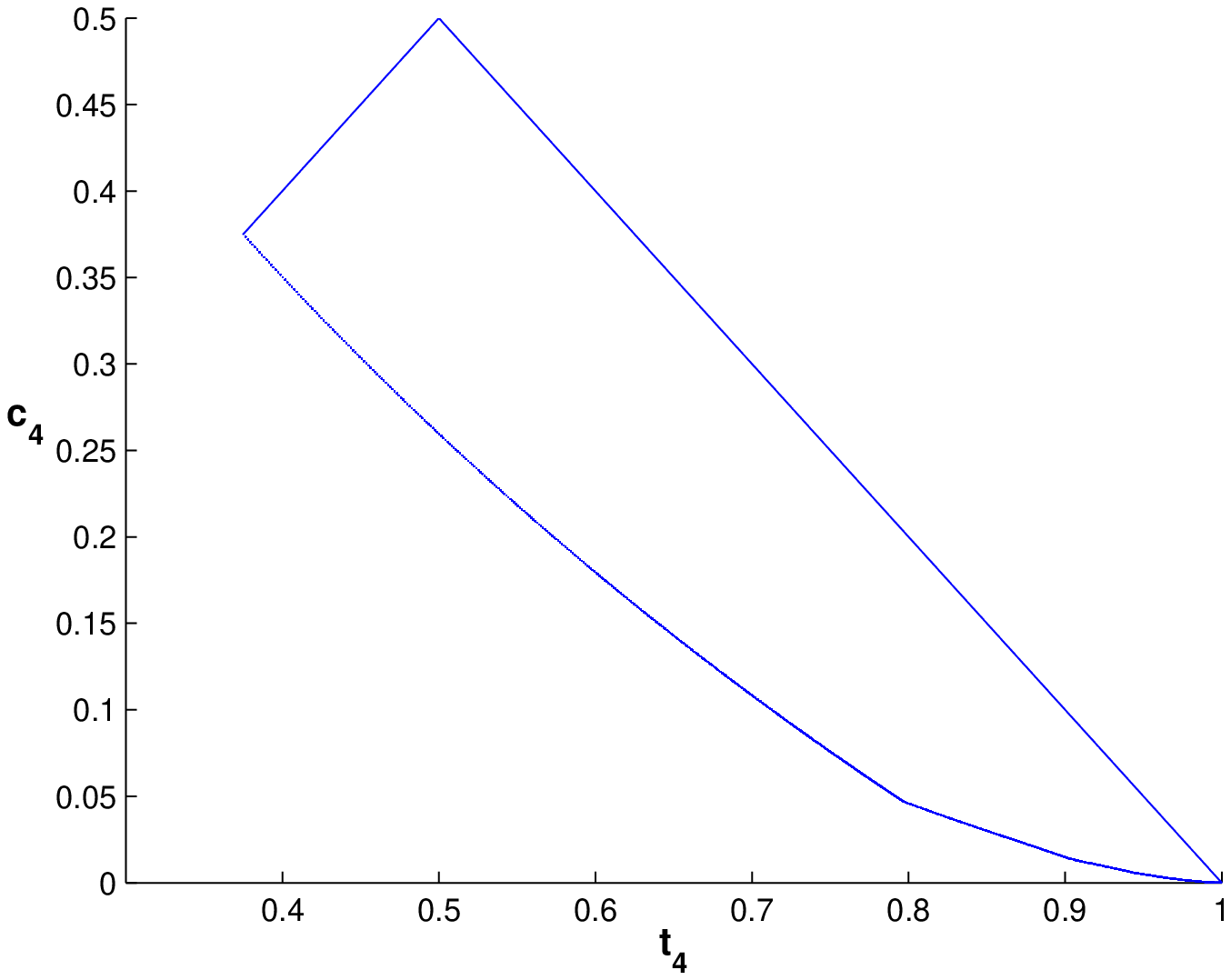}

  \caption{The boundary of the set $\{(t_4({\cal T}),c_4({\cal T}))\}$. The lower curve is conjectured.}

  \label{fig:3}

\end{figure}

We finally note that the planar sets $\{(t_3({\cal T}),t_4({\cal T}))\}$, $\{(c_3({\cal T}),c_4({\cal T}))\}$, $\{(t_4({\cal T}),c_4({\cal T}))\}$ are simply connected. This shows that these sets coincide with the bounded regions in \Cref{fig:1,fig:2,fig:3}.

\begin{lem}\label{lem:conv}

The set of all pairs $(c_3({\cal T}),c_4({\cal T}))$ is simply connected. Here ${\cal T}$ is an arbitrary family of tournaments for which these limits exist.

\end{lem}

The arguments introduced in Proposition~\ref{prop:equiv} yield the same conclusion for the sets $\{(t_3,t_4)\}$, $\{(t_4,c_4)\}$. We prove the lemma in Section~\ref{sect:proofs}.

\section{Results and conjectures}\label{sect:2}

\subsection{The maximum}

In this subsection we solve \Cref{it:a,it:b,it:c}:

\begin{observation}\label{obs1}

The following inequalities hold in all tournaments. These inequalities are tight.

\begin{itemize}

\item

$c_4\le 2 c_3$.

\item

$t_4\le 2 t_3 -1$.

\item

$c_4\le \min\{t_4,1-t_4\}$.

\end{itemize}

\end{observation}

\begin{proof}

Clearly $t_4+c_4\le 1$, since $t_4+c_4+w+l=1$. Also, as we saw $t_4-c_4=1-4c_3$ and $c_3+t_3=1$. In addition it is well-known and easy to show that $c_3\le \frac 14$. All the inequalities follow. To show that the first two inequalities are tight, we construct (Section~\ref{sect:proofs}) tournaments with $w=l=0$ for all values of $c_3\le \frac 14$. This  shows as well the tightness of the third inequality when $t_4\ge \frac 12$. For $t_4\le \frac 12$ we need a different construction which satisfies $c_4=t_4$, which we also do in Section~\ref{sect:proofs}. 

\end{proof}

\subsection{The minimum}

\Cref{it:d,it:e,it:f} are more involved. We focus on \Cref{it:d}. To derive an upper bound for \Cref{it:d} we introduce the {\em random blow-up} of a $k$-vertex tournament $H$. Associated with $H$ and a probability vector $(w_1,\ldots,w_m)$ is an infinite family of tournaments ${\cal T}={\cal T}(H;w_1,\ldots,w_m)$ whose $n$-th member has vertex set $\bigcupdot \{V_i | i\in H\}$ where $|V_i| = \lfloor w_i n \rfloor$. If $(i\to j)\in E(H)$ then there is an edge $(u\to v)$ from every $u\in V_i$ to every $v\in V_j$. The subtournament on each $V_i$ is random. In the {\em balanced} case $w_1=w_2=\ldots=w_k=\frac 1m$, we use the shorthand ${\cal T}(H)$.

We can now state our conjecture.

\begin{conj}\label{conj:1}

The minimum of $c_4$, given $c_3$ is attained by a random blow-up of a transitive tournament $T_m$.\\

Lemma~\ref{lem:blowup} below says that among all such tournaments of given $c_3$, the smallest $c_4$ is attained by taking $m$ as small as possible and $w_1=w_2=\ldots=w_{m-1}\ge w_m$.

\end{conj}

When $c_3=\frac{1}{4r^2}$ the random blow-up that minimizes $c_4$ is the balanced blow-up of $T_r$. It is conceivable that this case of the conjecture should be easier to handle. When $c_3=\frac 14$ and $m=1$ this reduces to the well-known fact that $t_4$ is minimized by a random tournament. (Recall that given $c_3$, minimization of $c_4$ and of $t_4$ are equivalent). In this article we settle as well the case $c_3=\frac{1}{16}$ and $m=2$.

Note that, if the conjecture is indeed true, then there is no simple expression for $\min c_4$ in terms of $c_3$.

\begin{rem}

The problem and the structure of the construction in Conjecture~\ref{conj:1} resemble a well studied problem in graph theory. The problem statement is as follows: For $r>s$ and a family of graphs with given $K_s$-density (that is the asymptotic probability for $s$ randomly chosen vertices to form a clique), how small can the $K_r$ density be?

The recent works of Razborov~\cite{raz2}, Nikiforov~\cite{nikiforov} and Reiher~\cite{reiher} solved the problem for $s=2$ (with $r=3$, $r=4$ and all $r\ge 5$, respectively). For all we know the problem is still open for $s=3$ and any $r$.

The optimal construction (Razborov~\cite{raz2}) for $s=2$, $r=3$ is obtained by blowup of complete graphs of the smallest possible size. The blowup weights are all equal except (maybe) one smaller part. There are no edges inside the blowup sets (unlike our random tournament inside each set).

\end{rem}

We reproduce a proof from~\cite{LM}, that we later (Lemma~\ref{lem:1}) improve.

\begin{prop}

$c_4\ge 6c_3^2$. 

\end{prop}

\begin{proof}

For an edge $e=uv$ in a tournament $T$, let $x_e$ be the probability that the triangle $uvw$ is cyclic when the vertex $w$ is selected uniformly at random. We define the random variable $X$ on $E(T)$ with uniform distribution that takes the value $x_e$ at $e\in E(T)$. Clearly $\mathbb{E} X = c_3+o_{|T|}(1)$. But a $4$-vertex tournament is isomorphic to $C_4$ iff it contains two cyclic triangles with a common edge. Consequently, $\mathbb{E}(X^2)=\frac{c_4}{6}+o_{|T|}(1)$. The proposition simply says that $Var(X)\ge 0$. \end{proof}

Consequently, our main problem is to find the smallest possible variance $Var(X)$ for given $\mathbb{E}(X)$. Conjecture~\ref{conj:1} and Lemma~\ref{lem:1} below are some quantitative forms of the assertion that when $0<c_3<\frac 14$, cyclic triangles cannot be uniformly distributed among the edges. We presently have no conceptual proof of this claim, and we must resort to flag algebra methods, which unfortunately offer no intuition as to the reason that this statement is true.

Here is another curious aspect of this problem. Define $\varphi_T(x):=\mathbf{pr}(X\ge x)$  and let $f:=\limsup_{|T|\to\infty}\varphi_T$. For all we know, $f$ may be discontinuous. To see this note that $f(\frac 13)\ge \frac 23$ where $\frac 23$ is the value that is attained by balanced blow-ups of $C_3$. We suspect that $f(\frac 13+)$ is strictly smaller than $\frac 23$. In fact, the best lower bound that we have is $f(\frac 13+)\ge \frac 49$ which is attained by an imbalanced blow-up of $C_3$ (e.g. a $(\frac 13+\epsilon, \frac 13+\epsilon, \frac 13-2\epsilon)$ blowup with $\epsilon\to 0^+$).

We turn next to apply Razborov's flag-algebra method~\cite{Raz} which yields a lower bound that is not far from the conjectured value. In particular, it proves Conjecture~\ref{conj:1} for $c_3=\frac 1{16}$.

\begin{lem}\label{lem:1}

$c_4\ge \frac{18c_3^2}{1+8c_3}$.

\end{lem}

\begin{cor}

If $c_3\ge \frac 1{16}$ then $c_4\ge \frac 3{64}$. In particular, the construction in Conjecture~\ref{conj:1} is optimal for $c_3=\frac 1{16}$.

\end{cor}

See ~\Cref{fig:4} for a comparison between the bound in Lemma~\ref{lem:1} and Conjecture~\ref{conj:1}.

Using available computer software, we were able to get further numerical evidence which indicates that Lemma~\ref{lem:1} is not tight for $c_3\neq 0,\frac 1{16},\frac 14$, and the true minimum of $c_4$ is closer to the conjectured value. The results are graphically presented in~\Cref{fig:4} and the method of computation is explained in Appendix~\ref{app:lbounds}.

\begin{figure}[!ht]

  \centering

  \includegraphics[width=160mm,height=100mm]{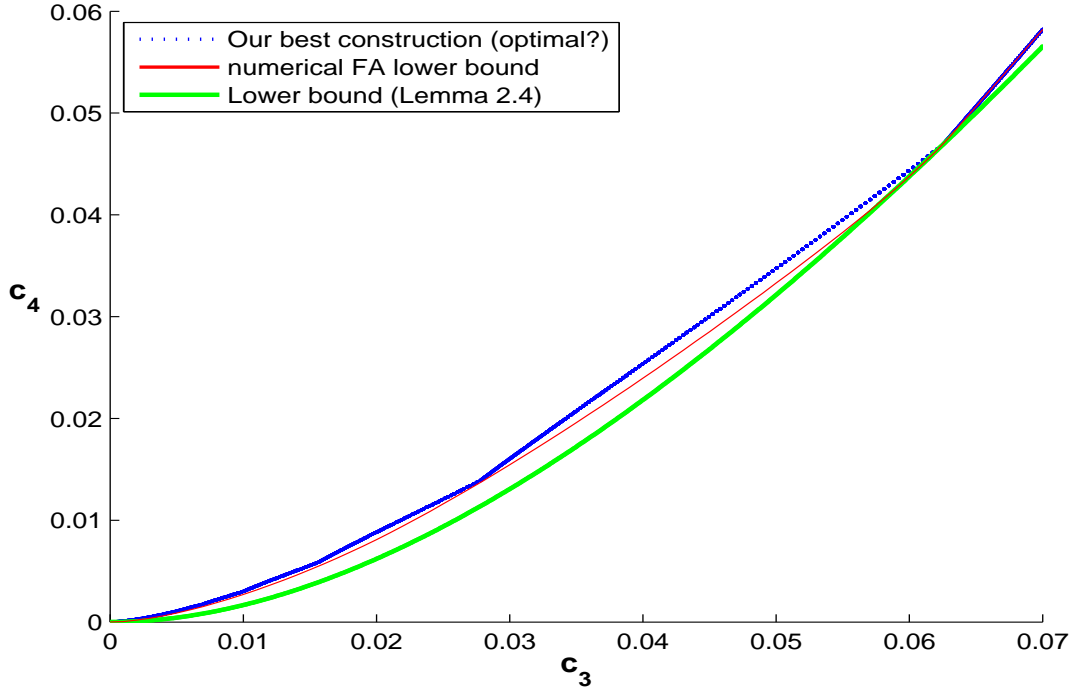}

  \caption{Lower bound from numerical application of flag algebras, compared with proven lower bound and our best construction, which is conjectured to be optimal. To improve visibility we present only the range $c_3\in [0,0.07]$. Numerical lower bound and the construction seem to coincide for $c_3\ge\frac 1{16}$.}

  \label{fig:4}

\end{figure}

Concluding this section, we formulate the following analytic lemma. It states that among all blow-ups considered in Conjecture~\ref{conj:1} the best one is a blow-up of a transitive tournament of least possible order, with equal vertex weights, except possibly one smaller weight.

\begin{lem}\label{lem:blowup}

Fix any $0<C<1$ and consider all probability vectors $w$ satisfying $\sum w_i^3=C$. The minimum of $\sum w_i^4$ among such vectors is attained by letting $w_1=\ldots=w_{m-1}\ge w_m>0$ with the smallest possible $m$.

\end{lem}

The relevance of the lemma in the setting of Conjecture~\ref{conj:1}, is that $c_3({\cal T})=\frac 14 \sum w_i^3$, and $c_4({\cal T})=\frac 38 \sum w_i^4$ where ${\cal T}={\cal T}(T_m;w_1,\ldots,w_k)$.

\section{Proofs}\label{sect:proofs}

\begin{proof}[Proof of Lemma~\ref{lem:conv}]

We will show that the set $\{(c_3,c_4)\}$ is vertically convex. Let ${\cal T}_1, {\cal T}_2$ be two families with $c_3({\cal T}_1)=c_3({\cal T}_2)$ and $c_4({\cal T}_1)<c<c_4({\cal T}_2)$. We construct an $n$-vertex tournament $T$ with $c_3(T)=c_3({\cal T}_1)+o_n(1)$ and $c_4(T)=c+o_n(1)$. Let $0\le p,\alpha \le 1$ be two constant parameters. Choose $T_1\in{\cal T}_1$ on $\alpha n$ vertices (we can choose a random subtournament of a larger member if ${\cal T}_1$ has no member of this order). Let $T_2\in{\cal T}_2$ of order $(1-\alpha)n$. Let $T=T_1\cupdot T_2$, where for $x\in T_1$ and $y\in T_2$ there is an edge $x\to y$ with probability $p$ and $y\to x$ with probability $1-p$.

We compute $c_3(T)=\alpha^3 c_3(T_1) + (1-\alpha)^3 c_3(T_2) + 3\alpha(1-\alpha)p(1-p)+o(1)$. Choose $p$ such that $p(1-p)=c_3({\cal T}_1)=c_3({\cal T}_2)$ and then $c_3(T)=c_3({\cal T}_1)+o(1)$.

In computing $c_4(T)$, several terms come in, each up to $+o(1)$ error

\begin{itemize}

\item

$\alpha^4 c_4(T_1)$ for quadruples contained in $T_1$

\item

$(1-\alpha)^4 c_4(T_2)$ all in $T_2$

\item

$6\alpha^2(1-\alpha)^2(p(1-p)+2p^2(1-p)^2)$ two in each

\item

$4\alpha^3(1-\alpha)(c_3({\cal T}_1)3p(1-p)+(1-c_3({\cal T}_1))p(1-p))$ three in $T_1$ and one in $T_2$

\item

$4\alpha(1-\alpha)^3(c_3({\cal T}_2)3p(1-p)+(1-c_3({\cal T}_2))p(1-p))$ one and three.

\end{itemize}

Consequently $c_4(T)$ is expressed (up to an additive $o(1)$ term) as a degree four polynomial in $\alpha$ which for $1 \ge \alpha \ge 0$ takes every value between $c_4({\cal T}_1)$ and $c_4({\cal T}_2)$.

\end{proof}

\begin{proof}[Completing the proof of tightness in Observation~\ref{obs1}]

Let us recall the well-known {\em cyclic tournaments} (see e.g.,~\cite{LM}). Place an odd number of vertices equally spaced along a circle, and $x\to y$ is an edge if the shorter arc from $x$ to $y$ is clock-wise. We are now ready to construct tournaments with the desired parameters.

\begin{itemize}

\item Tournaments with arbitrary $0\le c_3\le \frac 14$, and $w=l=0$:\\

Fix some $\frac n2 \le s \le n$. Let $T$ be the tournament with vertex set $1,2,\ldots,n$, where $x\to y$  for $1\le x < y\le n$, iff $y\le x+s$. 

We claim that $w(T)=l(T)=0$. For suppose that $x\to y\to z\to x$ is a cyclic triangle in $T$ and there is some vertex $w$ with either $w\to x,y,z$ or $w\leftarrow x,y,z$. w.l.o.g. $x<y,z$ and it follows that $x<y\le x+s<z\le y+s$. If $w<x$, then $w\to x$ since $s\ge \frac{n}{2}$, but $z\to w$. Likewise we rule out the possibility that $w>z$, i.e., necessarily $x<w<z$. If $x<w<y$ then necessarily $x\to w\to y$. Likewise, $y<w<z$ implies $y\to w\to z$.

For $n\to\infty$ odd and $s=\frac{n}{2}$ this yields the cyclic tournaments and $c_3=\frac 14$. when $s=n$ we obtain transitive tournaments. As $s$ varies

we cover the whole range $0\le c_3\le \frac 14$.

\item For $t\in[\frac 38,\frac 12]$, we construct a family $\cal T$ with $t_4({\cal T})=c_4({\cal T})=t$ (recall that no family of tournaments can have $t_4<\frac 38$): \\

Fix some $0\le p\le \frac 12$. We construct $\cal T$ from the cyclic tournaments by flipping each edge independently with probability $p$. As we show below, $c_3({\cal T})=\frac 14$, so by Observation~\ref{obs:43}, $t_4({\cal T})=c_4({\cal T})$. When $p=0$ we have the cyclic tournament with $t_4=c_4=\frac 12$ and when $p=\frac 12$ we have a random tournament with $t_4=c_4=\frac 38$. The claim follows by continuity.

To see that $c_3({\cal T})=\frac{1}{4}$, note that almost surely all vertex outdegrees in $T\in{\cal T}$ equal $\frac n2 \pm o(n)$. The claim follow by a standard Goodman-type argument.

\qed\end{itemize}

\let\qed\relax\end{proof}

\begin{proof}[Proof of Lemma~\ref{lem:1}]

We define the random variables $X$ and $Y$ over $E(T)$ with uniform distribution.

For $e=\{v_1\to v_2\}\in E(T)$ we define:

\begin{itemize}

\item $X(e)$ is the probability that $\{v_1,v_2,v_3\}$ is a cyclic triangle in $T$, where the vertex $v_3$ is chosen uniformly at random.

\item $Y(e)$ the probability that $\{v_1\to v_3\}\in E(T)$ and $\{v_3\to v_2\}\in E(T)$, where the vertex $v_3$ is chosen uniformly at random.

\end{itemize}

It is not hard to verify the following expectations: $\mathbb{E}(X)=c_3$, $\mathbb{E}(Y)=\frac {t_3}3$. $\mathbb{E}(X^2)=\frac {c_4}6$, $\mathbb{E}(Y^2)=\frac{t_4}6$ and $\mathbb{E}(X\cdot Y)=\frac{c_4}6$.

We define $Z=1+2(X-Y)$ and conclude that $\mathbb{E}(Z^2)=\frac{1+8c_3}3$ and $\mathbb{E}(X\cdot Z)=c_3$. By Cauchy-Schwarz $c_3^2=\mathbb{E}^2(X\cdot Z)\le \mathbb{E}(X^2)\mathbb{E}(Z^2)=c_4\frac{1+8c_3}{18}$.

\end{proof}

\begin{rem}

Proper disclosure: The above derivation could not have been carried out without seeing what flag-algebra calculations yield. 

\end{rem}

\begin{proof}[Proof of Lemma~\ref{lem:blowup}]

We wish to minimize $\sum_1^m w_i^4$ under the constraints $\sum_1^m w_i=1$ and $\sum_1^m w_i^3=C$ (for given $m$). We assume that all $w_i$ are positive, since zero $w_i$'s can be removed with smaller $m$. A Lagrange multipliers calculation yields that $w_i^3=\lambda w_i^2+\mu$ for all $i$ and for some constants $\lambda$ and $\mu$. The cubic polynomial $x^3-\lambda x^2-\mu$ has at most two positive roots since the linear term in $x$ vanishes. Therefore the coordinates of the optimal $w$ must take at most two distinct values.

Assume towards contradiction that $x>y>0$ appear as coordinates in $w$ with $y$ repeated at least twice. We will replace three of $w$'s coordinates $(x,y,y)$ while preserving $\sum_1^m w_i$ and $\sum_1^m w_i^3$, and reducing $\sum_1^m w_i^4=C$.

We replace $(x,y,y)$ by either $(s,t,0)$ or $(s,s,t)$ where $s\ge t\ge 0$. 

First, if $y\le \frac{\sqrt{5}-1}{4}\cdot x$, we prove the existence of $s\ge t\ge0$ s.t.

\begin{itemize}

\item $x+2y=s+t$.

\item $x^3+2y^3=s^3+t^3$.

\item $x^4+2y^4>s^4+t^4$.

\end{itemize}

Substitute $t=x+2y-s$ in the second equation: $x^3+2y^3=s^3+(x+2y-s)^3$, which can be rewritten as $(x+2y)s^2-(x+2y)^2s+2y(x+y)^2=0$. This quadratic has real roots iff $D=(x+2y)^4-8y(x+y)^2(x+2y)\ge 0$ which holds iff $x\ge (1+\sqrt{5}) y$, the range that we consider. Moreover, when $D\ge 0$, both roots are positive, since the quadratic has a positive constant term and a negative linear term. This proves the existence of $s\ge t\ge 0$ satisfying the first two conditions.

The sum of the fourth powers of the roots of this quadratic is $s^4+t^4=\frac{(x+2y)^8+6(x+2y)^4D+D^2}{8(x+2y)^4}$. Thus, it suffices to show that $8(x^4+2y^4)(x+2y)^2>8(x+2y)^6-64y(x+2y)^3(x+y)^2+64y^2(x+y)^4$ which is easily verified by expanding all terms.

In the complementary range $x>y\ge \frac{\sqrt{5}-1}{4}\cdot x$ we find $s\ge t\ge0$ s.t.

\begin{itemize}

\item $x+2y=2s+t$.

\item $x^3+2y^3=2s^3+t^3$.

\item $x^4+2y^4>2s^4+t^4$.

\end{itemize}

We substitute $t=x+2y-2s$ in the second equation: $x^3+2y^3=2s^3+(x+2y-2s)^3$, or, equivalently, $0=s^3-2(x+2y)s^2+(x+2y)^2s-y(x+y)^2=(s-y)(s^2-2xs-3ys+(x+y)^2)$. Thus, $s^2-(2x+3y)s+(x+y)^2=0$, 

and $s=\frac{2x+3y-\sqrt{y(4x+5y)}}{2}>0$. Now $t=x+2y-2s\ge 0$ iff $y\ge \frac{\sqrt{5}-1}{4}\cdot x$. Clearly, $s\ge t$.

It remains to compute

$$x^4+2y^4-2s^4-t^4=(8x^3+50x^2y+86xy^2+45y^3)\sqrt{y(4x+5y)}-(2x^4+52x^3y+180x^2y^2+232xy^3+101y^4)$$ 

and show that this is positive. To this end we must prove that 

$$y(4x+5y)(8x^3+50x^2y+86xy^2+45y^3)^2-(2x^4+52x^3y+180x^2y^2+232xy^3+101y^4)^2>0.$$

This expression can be written as

$$4(x-y)^3(19y^5+73xy^4+98x^2y^3+54x^3y^2+9x^4y-x^5)$$

which is positive, since $9x^4y-x^5>0$.

\end{proof}

\section{Further Directions}\label{sect:further}

\begin{itemize}

\item

Many basic open questions on the local profiles of combinatorial objects are still open. Thus, it is still unknown whether the set of $k$-profiles of graphs is a simply connected set. Similar issues were already raised in the pioneering work of Erd\H{o}s, Lov\'asz and Spencer~\cite{ELS}, and remains open. The analogous question for tournaments is open as well. We do not know if the set of $k$-local profiles of tournaments is {\em convex}. We don't know it even for $k=4$, and we are not sure what the right guess is. As first observed in~\cite{ELS} the analogous question is answered negatively for graphs. On the other hand, for trees the answer is positive~\cite{bubeck:linial}.

\item

We recall the random variable $X$ - the fraction of cyclic triangles containing a randomly chosen edge. It would be desirable to give a direct proof that $\text{Var}(X)>0$ for all $0<c_3<\frac 14$.

\item

In Section~\ref{sect:2} we defined the function $f(x)=\limsup_{|T|\to\infty}\mathbf{pr}(X\ge x)$. What can be said about $f$? In particular, is it continuous? Is it continuous at $\frac 13$?

\item

The conjectured extreme construction for~\Cref{it:d} is particularly simple when $c_3=\frac{1}{4k^2}$ for integer $k$. We were able to settle this case for $k=1,2$. Thus, the first open case is $c_3=\frac 1{36}$.

\item

To what extent can the lower bound in Lemma~\ref{lem:1} be improved using higher order flags? In particular, Figure~\ref{fig:4} suggests that our construction is optimal for $c_3\ge\frac{1}{16}$. Can the optimum for this range be established using flags of order $6$?

\item

Here we have studied the set $\{(t_3({\cal T}),t_4({\cal T}))\}$. We would like to understand the relationships among higher $t_k$'s as well.

\item

Obviously, we would be interested in further describing the set of $4$-profiles of tournaments.

\item

The powerful method of flag algebras remains mysterious, and it would be desirable to have more transparent local methods. Lemma~\ref{lem:1} and the stronger Conjecture~\ref{conj:1} offer concrete challenges for such methods.

\item

Associated with every tournament $T$ is a $3$-uniform hypergraph whose faces are the cyclic triangles of $T$. This hypergraph clearly does not contain a $4$-vertex clique and this was used in~\cite{er-haj:72} to deduce a lower bound on some hypergraph Ramsey numbers. We wonder about additional structural properties of such $3$-uniform hypergraphs. Specifically,

\begin{itemize}

\item

Can such hypergraphs be recognized in polynomial time?

\end{itemize}

\item

Lemma~\ref{lem:blowup} is the case $p=3, q=4$ of the following natural sounding question. Find the smallest $q$-norm among all probability vectors $w$ of given $p$-norm, where $q>p\ge 2$ are integers. Is it true that all optimal vectors have the form $w_1=\ldots=w_{m-1}\ge w_m$, with the least possible $m$? Clearly our method of proof is too ad-hoc to apply in general.

\end{itemize}

\section{Acknowledgement}

We are grateful to Yuval Peled for many useful discussions about flag algebras and for his help with the computer calculations described in Appendix~\ref{app:lbounds}. Thanks are also due to Gideon Schechtman for discussions concerning the proof of Lemma~\ref{lem:blowup}.

\bibliographystyle{amsplain}

\appendix

\section{Computer generated lower bounds}\label{app:lbounds}

We have used the flag algebra method as explained in Section $4$ of~\cite{yuval}. Using flags of size $3$ over the (only) type of order $2$ yields Lemma~\ref{lem:1}. Using flags of size $4$ over the same type we get a $16\times 16$ PSD matrix whose entries are bilinear expressions in the coordinates of a large tournament's $6$-profile. We used the cvx SDP-solver~\cite{cvx} to obtain the results presented in~\Cref{fig:4}. Working with larger flags may clearly yield better estimates but limited computational resources have stopped from reaching beyond size $4$-flags.

\end{document}